\documentclass[12pt,reqno]{amsart}
\usepackage{fullpage}
\usepackage{times}
\usepackage{amsmath,amssymb,amsthm,url}
\usepackage[utf8]{inputenc}
\usepackage[english]{babel}
\usepackage{bbm}
\usepackage{enumerate}
\usepackage{bm}
\usepackage{graphicx}
\usepackage{mathrsfs}
\usepackage[colorlinks=true, pdfstartview=FitH, linkcolor=blue, citecolor=blue, urlcolor=blue]{hyperref}
\usepackage{tikz-cd}
\usepackage{tabstackengine}
\stackMath
\usepackage{comment}

\newtheorem{thm}{Theorem}[section]
\newtheorem{lem}[thm]{Lemma}
\newtheorem{prop}[thm]{Proposition}
\newtheorem{cor}[thm]{Corollary}

\theoremstyle{definition}

\newtheorem{rem}[thm]{Remark}
\newtheorem{ex}[thm]{Example}

\newcommand{\Z}{\mathbb{Z}}
\newcommand{\F}{\mathbb{F}}

\title{Exact values and improved bounds on the clique number of cyclotomic graphs}
\author{Chi Hoi Yip}
\address{School of Mathematics\\ Georgia Institute of Technology\\ Atlanta, GA 30332\\ United States}
\email{cyip30@gatech.edu}
\begin{document}

\begin{abstract}
Let $q$ be an odd power of a prime $p$, and $S \subseteq \F_q^*$ such that $S=-S$ and $S/S \neq \F_q^*$. We show that the clique number of the Cayley graph $\operatorname{Cay}(\F_q^+,S)$ is at most $\sqrt{|S/S|}+\sqrt{q/p}$, improving the best-known $\sqrt{q}$ upper bound for many families of such graphs substantially. Such a new bound is strongest for cyclotomic graphs and in particular, it implies the first nontrivial upper bound on the clique number of all generalized Paley graphs of non-square order, extending the work of Hanson and Petridis. Moreover, our new bound is asymptotically sharp for an infinite family of generalized Paley graphs, and we further discover the first nontrivial family among them for which the clique number can be exactly determined. We also obtain a new lower bound on the number of directions determined by a large Cartesian product in the affine Galois plane $AG(2,q)$, which is sharp for infinite families.
\end{abstract}

\maketitle

\section{Introduction}

Throughout the paper, we let $p$ be a prime, $q$ a power of $p$, $\F_q$ the finite field with $q$ elements. We let $\F_q^+$ be the additive group of $\F_q$, and $\F_q^*=\F_q \setminus \{0\}$. 

A {\em clique} in a graph $X$ is a subgraph of $X$ that is a complete graph.  A {\em maximum clique} is a clique with the maximum size, while a {\em maximal clique} is a clique where one cannot add another vertex to it and still have a clique. For a graph $X$, the {\em clique number} of $X$, denoted $\omega (X)$, is the size of a maximum clique of $X$. 

This paper aims to present new progress toward the estimation of the clique number of cyclotomic graphs, including generalized Paley graphs. We begin by recalling a few basic terminologies. Let $S$ be a subset of $\F_q^*$ with $S=-S$, the \emph{Cayley graph} $X=\operatorname{Cay}(\F_q^+;S)$ is the graph with vertex set $\F_q$, such that two vertices are adjacent if and only if their difference is in the \emph{connection set} $S$. We are mostly interested in the case that $X$ is a \emph{cyclotomic graph}, namely, $S$ is a union of \emph{cyclotomic classes}, that is, $S$ is a union of cosets of a multiplicative subgroup of $\F_q^*$. Note that the clique number of $X$ is the maximum size of a subset $C$ of $\F_q$ such that $C-C \subseteq S \cup \{0\}$. Thus, estimating the clique number of $X$ is closely related to the additive structure of $S$, and such a problem is central in arithmetic combinatorics especially when $S$ is known to have some multiplicative structure \cite{CL07}. 

Cyclotomic graphs are well-studied, especially because they build a bridge among algebraic graph theory, arithmetic combinatorics, number theory, finite geometry, and other subjects (see for example \cite{AY22, BWX99, DSW, FX12, HP, Y22, Y25+}). Among cyclotomic graphs, generalized Paley graphs are of special interest. Let $d>1$ be a positive integer and $q$ be a prime power such that $q\equiv 1 \pmod {d}$; if $q$ is odd, assume further that $q \equiv 1 \pmod {2d}$ \footnote{This assumption guarantees that $-1$ is a $d$-th power in $\F_q$ so that $GP(q,d)$ is undirected.}. The {\em $d$-Paley graph} on $\F_q$, denoted $GP(q,d)$, is the Cayley graph $\operatorname{Cay}(\F_q^+;(\F_q^*)^d)$, where $(\F_q^*)^d=\{x^d: x \in \F_q^*\}$ is the set of $d$-th powers in $\F_q^*$. It is well-known that $\omega\big(GP(q,d)\big) \leq \sqrt{q}$; see for example \cite[Theorem 1.1]{Yip4}. This square root upper bound on the clique number is often referred to as the {\em trivial upper bound} in the literature. In fact, one of the folklore proofs \cite[Lemma 1.2]{Yip1} of such a square root upper bound extends to a larger family of Cayley graphs, which we describe below. 
\begin{lem}[Trivial upper bound]\label{lem:tub}
Let $q$ be a prime power. Let $S \subseteq \F_q^*$ with $S=-S$ such that $S/S \neq \F_q^*$. Then $\omega(\operatorname{Cay}(\F_q^+;S))\leq \sqrt{q}$.
\end{lem} 
\begin{proof}
Let $C=\{v_1, v_2, \ldots, v_N\} \subseteq \F_q$ be a maximum clique in $\operatorname{Cay}(\F_q^+;S)$. Let $g \in \F_q^* \setminus (S/S)$, and consider the set $W=\{v_i+gv_j: 1 \leq i,j \leq N\}$. It suffices to show that $|W|=N^2$ since $W \subseteq \F_q$. Assume that there exist $1 \leq i,i',j,j' \leq N$ such that $v_i+gv_j=v_i'+gv_j'$, which implies that $v_i-v_i'=g(v_j'-v_j)$. However, note that $v_i-v_i', v_j'-v_j \in S \cup \{0\}$ and $g \notin S/S$, forcing that $v_i- v_i=v_j'-v_j=0$, that is, $i=i'$ and $j=j'$.
\end{proof}

This trivial upper bound can sometimes be achieved, for example, if $q$ is a square and $\F_{\sqrt{q}}^* \subseteq S$, then $\F_{\sqrt{q}}$ forms a clique in the corresponding Cayley graph. In particular, if $q$ is a square and $d \geq 2$ is a divisor of $(\sqrt{q}+1)$, then $\omega(GP(q,d))=\sqrt{q}$ \cite{BDR88} (the converse is also true \cite{Yip4, Y25+}).

Even if $q$ is a non-square, this $\sqrt{q}$ trivial upper bound is the best-known for many families of Cayley graphs. We provide a significant improvement on this bound.

\begin{thm}\label{thm:Cayley}
Let $q$ be an odd power of a prime $p$. Let $S \subseteq \F_q^*$ with $S=-S$ such that $S/S \neq \F_q^*$. Then $\omega(\operatorname{Cay}(\F_q^+;S))< \sqrt{|S/S|}+1$ if $q=p$ and $\omega(\operatorname{Cay}(\F_q^+;S))< \sqrt{|S/S|}+\sqrt{q/p}-1$ if $q \neq p$.
\end{thm}

Note that our new upper bound highly depends on the multiplicative structure of the connection set $S$. In the case that $S$ is a union of cyclotomic classes, it is easy to compute the size of the ratio set $S/S$. In particular, Theorem~\ref{thm:Cayley} quickly implies the following important corollary on the clique number of cyclotomic graphs. While a cyclotomic graph is simply the edge-disjoint union of copies of a given generalized Paley graph, in general it is more complicated to estimate its clique number.

\begin{cor}\label{cor:main}
Let $d \geq 2$, and let $q \equiv 1 \pmod {d}$ be an odd power of a prime $p$. Let $H=(\F_q^*)^d$, $g$ be a primitive root of $\F_q$, and $I \subseteq \Z/d\Z$ be an index set. Consider the cyclotomic graph $X=\operatorname{Cay}(\F_q^+, S_{d,I})$, where $S_{d, I}=\bigcup_{i \in I} g^i H$ satisfies $S_{d,I}=-S_{d,I}$. If $I-I \neq \Z/d\Z$, then 
$$
\omega(X) \leq \sqrt{\frac{|I-I|q}{d}}+\sqrt{\frac{q}{p}}.
$$    
In particular, assume additionally that $q\equiv 1 \pmod {2d}$ if $q$ is odd, $\omega(GP(q,d))\leq  \sqrt{q/d}+\sqrt{q/p}$. 
\end{cor}
\begin{proof}
Since $I-I \neq \Z/d\Z$, it follows that 
$$S_{d,I}/S_{d,I}=\bigcup_{i,j \in I} g^{i-j} H=\bigcup_{k \in I-I} g^k H \subsetneq \F_q^*.$$
The corollary thus follows from Theorem~\ref{thm:Cayley} immediately.
\end{proof}

Theorem~\ref{thm:Cayley} is inspired by the recent breakthrough of Hanson and Petridis \cite{HP} on the clique number of generalized Paley graphs. They used Stepanov's method \cite{S69} to show that $\omega\big(GP(p,d)\big)\leq \sqrt{p/d}+1$ when $q=p$ is a prime. Their method has been extended by the author to an arbitrary prime power \cite{Yip2, Yip1} in the following result:

\begin{thm}[{\cite{HP},\cite[Theorem 5.8]{Yip2}}] \label{thm:Nn}
Let $q$ be a power of an odd prime $p$. If $q\equiv 1 \pmod{2d}$, and $2 \leq n\leq N=\omega\big(GP(q,d)\big)$ satisfies
$
\binom{n-1+\frac{q-1}{d}}{\frac{q-1}{d}}\not \equiv 0 \pmod p,
$
then $(N-1)n \leq \frac{q-1}{d}$.
\end{thm}

When $q=p$ is a prime, it is easy to verify that the binomial coefficient condition is automatically satisfied. In general, Theorem~\ref{thm:Nn} provides an algorithm to give an upper bound on $\omega(GP(q,d))$ by analyzing the base-$p$ representation of $\frac{q-1}{d}$. When $d=2$, which corresponds to Paley graphs, the author \cite{Yip1} used this idea to prove that $\omega\big(GP(q,2)\big)\leq \sqrt{q/2}+O(\sqrt{q/p})$ when $q$ is an odd power of a prime $p$. More generally, when $d \mid (p-1)$ and $d \geq 3$, the author \cite[Section 5.2]{Yip2} proved that $\omega\big(GP(q,d)\big)\leq (1+o(1))\sqrt{q/d}$. However, as remarked in \cite[Section 5.2]{Yip2}, in general, the analysis is fairly complicated and it is not clear if this approach would always achieve a \emph{nontrivial} upper bound on the clique number. There are other non-trivial upper bounds on $\omega(GP(q,d))$ in \cite{Yip2}, but they are all of the shape $(1-o(1))\sqrt{q}$. To conclude, in general, the best-known upper bound on $\omega(GP(q,d))$ is $(1-o(1))\sqrt{q}$.

Our new results refine the Hanson-Petridis method and its generalization in the following aspects. Firstly, our results show that the barrier created by the binomial coefficients in Theorem~\ref{thm:Nn} can be essentially removed and thus we extend the Hanson-Petridis bound to all finite fields with non-square order. Note that even if $q$ is assumed to be a non-square, the non-vanishing condition on binomial coefficients cannot be completely removed; a family of counterexamples can be found in Theorem~\ref{thm:cliquenumber}. More importantly, it seems the Hanson-Petridis method and its generalization apply only to generalized Paley graphs, while our approach produces a similar upper bound on the clique number of cyclotomic graphs or general Cayley graphs. Interestingly, our proof is based on a simple variant of Stepanov's method; see Remark~\ref{rem:Stepanov}.

Let $p$ be a prime. Observe that when $q=p^3$ and $d \geq 2$ is a divisor of $p^2+p+1$, the subfield $\F_p$ forms a clique in $GP(q,d)$ and thus $\omega(GP(q,d))\geq p$; this was first observed by Broere, D\"{o}man, and Ridley \cite{BDR88}. It was recently proved that $\F_p$ forms a maximal clique in $GP(q,d)$ \cite{Y23} for sufficiently large $p$, and it is tempting to conjecture that $\omega(GP(q,d))=p$ since $\F_p$ is the only ``obvious" large clique. Indeed, when $p=o(d)$, Corollary~\ref{cor:main} implies that $\omega(GP(q,d))=(1+o(1))p$, as $p,d \to \infty$. This shows that Corollary~\ref{cor:main} is asymptotically sharp for an infinite family of graphs; this is rather surprising since the 
Hanson-Petridis bound, despite being the best-known and a highly nontrivial upper bound for $\omega(GP(p,d))$, is far from the conjectural bound $p^{o(1)}$ predicted from the Paley graph conjecture (see for example \cite[Section 2.2]{Yip4}). Inspired by this observation, we further discover that the clique number in this setting is in fact exactly $p$. 

\begin{thm}\label{thm:cliquenumber}
Let $p$ be an odd prime.
Let $q=p^3$ and $d$ be a divisor of $p^2+p+1$. If $d>p$, then $\omega(GP(q,d))=p$.  
\end{thm}

We remark that Theorem~\ref{thm:cliquenumber} is the first nontrivial instance for which the clique number of an infinite family of generalized Paley graphs can be determined, and it refines several results in \cite{BDR88, Y22, Y23}. More generally, if a generalized Paley graph $GP(q,d)$ admits a clique which is a subfield of $\F_q$ and $K$ is the largest such subfield, the author \cite{Y23} used character sum estimates to show that $K$ is a maximal clique. It is also tempting to conjecture that $\omega(GP(q,d))=|K|$, and we confirm that under some extra assumptions in Proposition~\ref{prop:cliquenumber2}. We also construct several interesting infinite families of such graphs in Section~\ref{sec:sec4}.

Theorem~\ref{thm:Cayley} is a consequence of Theorem~\ref{thm:mainD}, concerning a new lower bound on the number of directions determined by a Cartesian product. Before stating the new bound, we recall a few basic definitions. Let $AG(2,q)$ denote the {\em affine Galois plane} over the finite field $\F_q$. Let $U$ be a subset of points in $AG(2,q)$. We use Cartesian coordinates in $AG(2,q)$ so that $U=\{(x_i,y_i):1 \leq i \leq |U|\}$.
The set of {\em directions determined by} $U \subseteq AG(2, q)$ is 
\[ \mathcal{D}_U=\left\{ \frac{y_j-y_i}{x_j-x_i} \colon 1\leq i <j \leq |U| \right \} \subseteq \F_q \cup \{\infty\},\]
 where $\infty$ corresponds to the vertical direction. While the connection between the theory of directions and the clique number of generalized Paley graphs was only discovered in recent papers \cite{HP, Yip2}, the proof of Lemma~\ref{lem:tub} already indicates such a connection. 

R\'edei \cite{LR73} and Sz\H{o}nyi \cite{S99} proved upper bounds on the size of $\mathcal{D}_U$ for a general point set $U$. When the point set $U$ is a Cartesian product, Di Benedetto, Solymosi, and White \cite{DSW} proved the following improved upper bound.

\begin{thm}[\cite{DSW}]\label{DSW}
Let $p$ be a prime. Let $A, B \subseteq \F_p$ be sets each of size at least two such that $|A||B| < p$. Then the set of points $A\times B\subseteq AG(2,p)$ determines at least $|A||B| - \min\{|A|,|B|\} + 2$ directions.
\end{thm}

They observed that Theorem~\ref{DSW} can be used to recover the Hanson-Petridis bound. We refer to \cite{MWY22, S21} for improvements on Theorem~\ref{DSW} for the Cartesian product $\{1,2,\ldots, n\}^2$ as well as their applications to problems concerning Diophantine
equations. Unfortunately, Theorem~\ref{DSW} fails to extend to general finite fields due to the subfield obstruction. 
The author \cite[Theorem 1.6]{Yip2} extended Theorem~\ref{DSW} to general finite fields under extra assumptions. However, such extensions are not strong enough for the applications. Overcoming the subfield obstruction, we establish the following lower bound with a ``corrected error term" when the point set is a large Cartesian product. 

\begin{thm}\label{thm:mainD}
Let $q$ be a power of a prime $p$.
Let $A,B\subseteq \F_q$ be sets such that $|A|=m \geq 2, |B|=n \geq 2$, and $mn \leq q$. Then the set of points $A\times B\subseteq AG(2,q)$ determines at least $mn-\min \{p^{s_1}(n-1), p^{s_2}(m-1)\}+1$ directions, where $s_1$ is the largest integer such that $p^{s_1}n \leq q$ and $s_2$ is the largest integer such that $p^{s_2}m \leq q$.
\end{thm}

Theorem~\ref{thm:mainD} is a generalization of Theorem~\ref{DSW}. Indeed, when $q=p$ is a prime, we have $s_1=s_2=0$, and thus Theorem~\ref{thm:mainD} recovers Theorem~\ref{DSW}. Our proof is different from the approach by Di Benedetto, Solymosi, and White \cite{DSW}; see Remark~\ref{rem:Stepanov}. When $q$ is a square, and $A=B=\F_{\sqrt{q}}$ is given by the subfield with size $\sqrt{q}$, the direction set determined by $A \times B$ is $\F_{\sqrt{q}} \cup \{\infty\}$, and our bound gives $q-\sqrt{q}(\sqrt{q}-1)+1=\sqrt{q}+1$, which is sharp. More generally, if $B$ is a subfield of $\F_q$ and $A$ is a subspace over $B$ with $|A||B|=q$, our bound is also sharp. 

When $|U|\leq q$, the best-known lower bound on $|\mathcal{D}_U|$ is roughly of the form $|U|/\sqrt{q}$, due to Dona \cite{D21}.% (see also the related work by Fancsali, Sziklai, and Tak\'{a}ts \cite{FST13}). 
While the lower bound given by Theorem~\ref{thm:mainD} could be trivial when $A \times B$ is small, we remark that Theorem~\ref{thm:mainD} does improve Dona's bound significantly when $U$ is a large Cartesian product. For example, if $q$ is a non-square, and $\sqrt{q/p}=o(\min \{|A|, |B|\})$, then Theorem~\ref{thm:mainD} implies that $A \times B$ determines $(1-o(1))|A||B|$ directions. In particular, we have the following corollary. 

\begin{cor}\label{cor:sum-product}
Let $q=p^{2r+1}$, where $p$ is a prime and $r$ is a non-negative integer. If $A \subseteq \F_q$ such that $2p^r<|A|<\sqrt{q}$, then 
$$
\bigg|\frac{A-A}{A-A}\bigg| >\frac{|A|^2}{2}.
$$
\end{cor}
\begin{proof}
Let $U=A \times A \subset AG(2,q)$; then $|\mathcal{D}_U|=\frac{A-A}{A-A}$. Since $2p^r<|A|<\sqrt{q}$, Theorem~\ref{thm:mainD} (with $s_1=s_2=r$) implies that $|\mathcal{D}_U|>|A|^2-p^r|A|>|A|^2/2$.
\end{proof}

Results similar to Corollary~\ref{cor:sum-product} have played an important role in recent breakthroughs on sum-product type estimates over $\F_p$ (see for example \cite{MPRRS19, SS22}). It would be interesting to see if Corollary~\ref{cor:sum-product} would imply new sum-product type estimates over $\F_q$. 

\medskip

\textbf{Notation.} We follow standard notation for arithmetic operations among sets. Let $q$ be a prime power. Given two sets $A,B \subseteq \F_q$, we write $A-B=\{a-b: a \in A, b \in B\}$ and $A/B=\{a/b: a \in A, b \in B\}$, where $a/0$ is defined to be $\infty$ for each $a\in \F_q$.

%$A+B=\{a+b: a \in A, b \in B\}$, $A-B=\{a-b, a \in A, b \in B\}$, $AB=\{ab: a \in A, b \in B\}$ and $A/B=\{a/b: a \in A, b \in B\}$.

\textbf{Structure of the paper.} In Section~\ref{sec:sec2}, we prove Theorem~\ref{thm:mainD}. % and give a few interesting corollaries. 
In Section~\ref{sec:app}, we use Theorem~\ref{thm:mainD} to deduce Theorem~\ref{thm:Cayley}. In Section~\ref{sec:sec4}, we construct nontrivial families of generalized Paley graphs for which their clique number can be explicitly determined, and deduce Theorem~\ref{thm:cliquenumber} as a special case. 

\section{Improved lower bound on the size of direction sets}\label{sec:sec2}
Instead of introducing hyper-derivatives and binomial coefficients as in \cite{Yip2, Yip1}, to prove Theorem~\ref{thm:mainD}, we consider the following folklore lemma, which characterizes polynomials with derivative identically zero. %The lemma is well-known, and we include a simple proof.

\begin{lem}\label{lem:derivative0}
Let $q$ be a power of a prime $p$. Let $f(x) \in \F_q [x]$ be a nonzero polynomial such that its derivative $f'=0$. Then the multiplicity of each root of $f$ (in the algebraic closure $\overline{\F_q}$) is a multiple of $p$.
\end{lem}
\begin{comment}
\begin{proof}
Without loss of generality, we may assume $f$ is a monic polynomial. Let $\beta_1, \beta_2, \cdots, \beta_N$ be distinct roots of $f$ and factorize $f$ as 
$$
f(x)=\prod_{j=1}^N (x-\beta_j)^{\alpha_j},
$$
where $\alpha_j$ is the multiplicity of the root $\beta_j$. 

Then the derivative of $f$ can be computed as follows:
$$
f'(x)=\sum_{i=1}^N \alpha_i (x-\beta_i)^{\alpha_i-1} \bigg(\prod_{j \neq i}   (x-\beta_j)^{\alpha_j}\bigg)=\prod_{k=1}^{N} (x-\beta_k)^{\alpha_k-1} \cdot \sum_{i=1}^N \alpha_i \bigg(\prod_{j \neq i}   (x-\beta_j)\bigg).
$$
Since $f' \equiv 0$, it follows that 
$$
\sum_{i=1}^N \alpha_i \bigg(\prod_{j \neq i}   (x-\beta_j)\bigg) \equiv 0.
$$
In particular, for each $1 \leq i \leq N$, we must have $\alpha_i=0$ (in $\F_q$) by setting $x=\beta_i$. Therefore, the multiplicity of each root is a multiple of $p$.
\end{proof}
\end{comment}

Next, we present the proof of Theorem~\ref{thm:mainD} using R\'edei polynomial with Sz\H{o}nyi's extension, together with a simple variant of Stepanov's method (see Remark~\ref{rem:Stepanov}).

\begin{proof}[Proof of Theorem~\ref{thm:mainD}]
Let $A=\{a_1,a_2, \ldots,a_m\}$ and $B=\{b_1, b_2, \ldots, b_n\}$ be subsets of $\F_q$. Let $U=A \times B \subseteq AG(2,q)$. Note that the direction set determined by $U$ only depends on $A-A$ and $B-B$, so without loss of generality we can assume that $0 \in A, B$. It suffices to prove $|\mathcal{D}_U| \geq mn-p^{s_1}(n-1)+1$. For the other bound, it follows from switching the role of $A$ and $B$. 

The \emph{R\'edei polynomial} \cite{LR73} of $U$ is defined as 
$$ R(x,y) = \prod_{(a,b)\in U} (x+ay-b)=\prod_{i=1}^m \prod_{j=1}^n (x+a_iy-b_j) .$$
Note that if $y_0 \in \F_q \setminus \mathcal{D}_U$, then $ay_0-b \neq a'y_0-b'$ for any two distinct points $(a,b), (a',b')$ in $U$. Thus, if $y_0 \in \F_q \setminus \mathcal{D}_U$, then $R(x,y_0)$ divides $x^q-x$. Sz\H{o}nyi \cite{S96, S99} (see also \cite{DSW,Yip2}) showed that there exist a polynomial $F(x,y) \in \F_q[x,y]$ and polynomials $h_i(y) \in \F_q[y]$ with $\operatorname{deg} (h_i) \leq i$, such that 
\begin{equation}\label{eq1}
 R(x,y)F(x,y) = x^q + h_1(y)x^{q-1} + h_2(y)x^{q-2} + \cdots + h_q(y),
\end{equation}
and for each $y_0 \in \F_q \setminus \mathcal{D}_U$,
$$
R(x,y_0)F(x,y_0)=x^q-x.
$$
This special polynomial $F(x,y)$ is also known as \emph{Sz\H{o}nyi's extension polynomial}; the structure of $F(x,y)$ is complicated, we refer to \cite[Section 2.2]{Yip2} for an explicit formula of $F(x,y)$.

Let $c_i=h_{q-i}(0)$ for $0 \leq i \leq q-1$. Then equation~\eqref{eq1} implies that
\begin{equation} \label{eq2}
 G(x):=R(x,0)F(x,0) = F(x,0)\prod_{j=1}^n (x-b_j)^m =x^q+\sum_{i=0}^{q-1} c_ix^i.
\end{equation}
We claim that if $c_i \neq 0$ for some $0 \leq i \leq q-1$, then we have $|\mathcal{D}_U|\geq i+1$ \cite{LR73, S99}. We include a short proof for the sake of completeness.  Indeed, if $c_i \neq 0$, then $h_{q-i} \not 0$, which implies that $h_{q-i}(y_0)=0$ for at most $q-i$ distinct $y_0 \in \F_q$. However, we know that $h_{q-i}(y_0)=0$ for each $y_0 \in \F_q \setminus \mathcal{D}_U$. This shows that $|\F_q \setminus \mathcal{D}_U|\leq q-i$, and thus $|\mathcal{D}_U| \geq i+1$ (the extra $1$ comes from the infinity direction). 

Next we factorize $G(x)$ into linear factors over $\overline{\F_q}$: 
$$
G(x)=\prod_{j=1}^N (x-\beta_j)^{k_j},
$$
where $\beta_1, \beta_2, \ldots, \beta_N$ are the distinct roots of $G$, and $\beta_j=b_j$ for each $1 \leq j \leq n$. In particular, we have $k_j \geq m$ for each $1 \leq j \leq n$, $k_j \geq 1$ for each $n<j \leq N$, and $\sum_{j=1}^N k_j=q$. Let $t$ be the largest integer such that $p^t \mid k_j$ for all $1 \leq j \leq N$. Then we can write $k_j=p^t \alpha_j$ for each $1 \leq j \leq N$, with at least one $\alpha_j$ not divisible by $p$, so that
$$
G(x)=\prod_{j=1}^N (x-\beta_j)^{k_j}=\prod_{j=1}^N (x^{p^t}-\beta_j^{p^t})^{\alpha_j}.
$$

Observe that the map $x \mapsto x^{p^t}$ is an automorphism on $\F_q$ since $\gcd(p^t, q-1)=1$. Set $y=x^{p^t}$, then we have 
\begin{equation} \label{eq3}
H(y):=\prod_{j=1}^N (y-\beta_j^{p^t})^{\alpha_j}=G(x).
\end{equation}
Note that $\beta_j^{p^t}$ are still distinct since the map $\beta \mapsto \beta^{p^t}$ is also an automorphism on $\overline{\F_q}$. Since not all $\alpha_j$ are multiples of $p$, it follows from Lemma~\ref{lem:derivative0} that $H'$ is not identically zero. Thus, in view of the factorization given in equation~\eqref{eq3}, $H'$ has degree at least
$$
\sum_{j=1}^N (\alpha_j-1)
$$
since for each $1 \leq j \leq N$, the root $\beta_j^{p^t}$ is a root of $H'$ with multiplicity at least $\alpha_j-1$.

On the other hand, note that equation~\eqref{eq2} becomes: 
$$
H(y)=x^q+\sum_{i=0}^{q-1} c_ix^i=y^{q/p^t}+\sum_{i=0}^{q/p^t-1} d_iy^i
$$
as polynomials over $\F_q$, where $d_i=c_{p^ti}$ for each $0 \leq i \leq q/p^t-1$. Note that $q/p^t$ is a multiple of $p$ and thus
$$
H'(y)=\sum_{i=1}^{q/p^t-1} id_iy^{i-1}.
$$
It follows that there is $1 \leq i_0 \leq q/p^t-1$ such that $i_0d_{i_0} \neq 0$ and 
$$
i_0-1 \geq \sum_{j=1}^N (\alpha_j-1).
$$
Therefore, our earlier claim implies that
\begin{equation} \label{eq4}
|\mathcal{D}_U|\geq p^ti_0+1 \geq p^t \bigg(\sum_{j=1}^N (\alpha_j-1)+1\bigg)+1=\sum_{j=1}^N k_j -p^t (N-1)+1 =q-p^t(N-1)+1.
\end{equation} 

Recall that each $k_j$ is a multiple of $p^t$. In particular, $k_j \geq p^t$ for each $j>n$. It follows that
$$
N-n \leq \frac{q-\sum_{j=1}^{n} k_j}{p^t} \leq \frac{q-mn}{p^t},
$$
that is, $p^tN \leq q-mn+p^tn$. Putting this estimate into equation~\eqref{eq4}, we conclude that
$$
|\mathcal{D}_U|\geq q-p^tN+p^t+1\geq mn-p^t(n-1)+1.
$$
Finally, note that $t \leq s_1$, since $s_1$ is the largest integer such that $p^{s_1}n \leq q$. This completes the proof.
\end{proof}

\begin{rem}\label{rem:Stepanov}
Our proof differs from the original proof of Di Benedetto,  Solymosi, and White \cite{DSW} on Theorem~\ref{DSW} (which corresponds to the case that $q$ is a prime). In particular, \cite[Lemma 6]{DSW} plays an important role in their proof. Such a lemma has been extended to all finite fields in \cite[Lemma 3.1]{Yip2} by the author, which did lead to a generalization of their result (see \cite[Theorem 1.6]{Yip2}). However, such a generalization only yields a tiny improvement on the trivial upper bound on the clique number of generalized Paley graphs \cite[Theorem 1.9]{Yip2}.

We replaced this key lemma with a simple variant of Stepanov's method, inspired by the applications of Stepanov's method in \cite{Yip2, Yip1}. In particular, to apply Stepanov's method, one needs to find a low-degree polynomial that vanishes on each element of a desired set with high multiplicity, which is guaranteed by the Cartesian product structure of the point set. It is also crucial that the polynomial constructed is \emph{not identically zero}, which is ensured by the non-vanishing of binomial coefficients described in Theorem~\ref{thm:Nn} in \cite{Yip2, Yip1}. Here we use a different yet simpler criterion, namely, Lemma~\ref{lem:derivative0}, to achieve this purpose more effectively. 
\end{rem}
    
Next, we deduce two interesting corollaries of Theorem~\ref{thm:mainD}. 

The analogue of the Cauchy-Davenport theorem over finite fields (equivalently, vector spaces over prime fields) was established by Eliahou and Kervaire \cite{EK98}. We use Theorem~\ref{thm:mainD} to deduce a special case. 

\begin{cor}\label{cor:CD}
Let $q$ be a power of a prime $p$. Let $A \subseteq \F_q$ such that $|A|>q/p$. Then $|A-A|\geq \min \{2|A|-q/p, q\}$. In particular, if $q/p|A| \to 0$, then $|A-A|\geq \min \{(2-o(1))|A|, q\}$.
\end{cor}
\begin{proof}
Let $B=\{0,1\}$. Then the direction set determined by $A \times B$ is $(A-A) \cup \{\infty\}$. If $|A|<q/2$, then Theorem~\ref{thm:mainD} implies that the direction set has size at least $2|A|-q/p+1$ since $p^{s_1}=q/p$, and the lower bound on $|A-A|$ follows. If $|A|>q/2$, then the direction set determined by $A \times B$ is $\F_q \cup \{\infty\}$ by the pigeonhole principle \cite[Lemma 2.1]{D21}. 
\end{proof}

Identify $\F_q$ with $\F_p^{n}$. Corollary~\ref{cor:CD} is sometimes sharp. For example, if $p\geq 3$ and $A=\{0,1\} \times \F_p^{n-1}$, then $A-A=\{-1,0,1\} \times \F_p^{n-1}$. 

Let $K$ be a proper subfield of $\F_q$. We know if $A \subseteq \F_q$ is a vector space over $K$, then $(A-A)K=AK=A$ and thus $|AK|=|A|$. We use Theorem~\ref{thm:mainD} to deduce that $(A-A)K^*$ expands significantly as long as $|A| \geq q/|K|+1$. 

\begin{cor}
Let $q$ be a prime power. Let $K$ be a proper subfield of $\F_q$. If $A \subseteq \F_q$ such that $|A|\geq q/|K|+1$, then $|(A-A)K|>|A|(|K|-|K|/p-1)$.
\end{cor}

\begin{proof}
Let $B=K^*$. Then the direction set determined by $A \times B$ is given by $$\frac{A-A}{B-B}=\{\infty\} \cup\frac{A-A}{K^*}=\{\infty\} \cup(A-A)K.$$ Since $|A|\geq q/|K|+1$, it follows that $p^{s_2}\leq|K|/p$, where $s_2$ is the largest integer such that $p^{s_2}|A|\leq q$. Thus, Theorem~\ref{thm:mainD} implies that  
$$
|(A-A)K| \geq |A|(|K|-1)-|K|(|A|-1)/p>|A|(|K|-|K|/p-1),
$$ 
as required.
\end{proof}

\section{Applications to clique number of Cayley graphs}\label{sec:app}

In this section, we use Theorem~\ref{thm:mainD} to deduce Theorem~\ref{thm:Cayley} on the clique number of Cayley graphs.

\begin{proof}[Proof of Theorem~\ref{thm:Cayley}]
Let $q=p^{2r+1}$, where $r$ is a non-negative integer. Let $C$ be a maximum clique in $\operatorname{Cay}(\F_q^+;S)$ with $|C|=n$. We may assume that $n>p^{r}$, otherwise clearly we are done. 

Since $S/S \neq \F_q^*$, by Lemma~\ref{lem:tub}, $n \leq \sqrt{q}$. Consider the point set $U=C \times C \subseteq AG(2,q)$; then we have $|U|=n^2\leq q$. Note that the direction set determined by $U$ is
$$\mathcal{D}_U=\frac{C-C}{C-C} \subseteq \frac{S}{S} \cup \{0,\infty\}.$$
In particular, we have 
$$|\mathcal{D}_U|\leq |S/S|+2.$$ 

Let $s$ be the largest integer such that $p^sn \leq q$; then we have $s \leq r$. It follows from Theorem~\ref{thm:mainD} that 
$$
|\mathcal{D}_U| \geq n^2-p^s(n-1)+1 \geq n^2-p^r(n-1)+1.
$$
Comparing the above two estimates, we obtain that
$$
n^2-p^r(n-1)+1 \leq |S/S|+2,
$$
that is,
$$
n \leq \frac{p^r}{2}+\sqrt{|S/S|+\bigg(\frac{p^r}{2}-1\bigg)^2} < \frac{p^r}{2}+\sqrt{|S/S|}+\left|\frac{p^r}{2}-1\right|.
$$
When $r=0$, we obtain that $n<\sqrt{|S/S|}+1$; when $r \geq 1$, we obtain that $n<\sqrt{|S/S|}+p^r-1$. 
\end{proof}

\begin{rem}
For the clique number of $GP(q,d)$, our new bound in Corollary~\ref{cor:main} is slightly weaker than the upper bounds in \cite{Yip1} and \cite[Theorem 5.10]{Yip2} in the special case $d \mid (p-1)$ for which the base-$p$ representation of $\frac{q-1}{d}$ is simple and Theorem~\ref{thm:Nn} is the most effective. However, in the generic case, our new bound improves the best-known upper bound \cite[Theorem 6.5]{Yip2} substantially.
\end{rem}

\section{Generalized Paley graphs with exceptionally large cliques}\label{sec:sec4}

The next proposition allows us to determine the clique number of families of generalized Paley graphs with exceptionally large cliques. While it is not hard to deduce it from Theorem~\ref{thm:Nn}, it is rather surprising that Theorem~\ref{thm:Nn} could be used to determine the clique number, even for special generalized Paley graphs.

\begin{prop}\label{prop:cliquenumber2}
Let $q$ be a power of an odd prime $p$ and let $K$ be a proper subfield of $\F_q$. Let $d \geq 2$ be a divisor of $\frac{q-1}{|K|-1}$ such that $q< d|K|(|K|+1)$. Let $r$ be the remainder of $\frac{q-1}{d}$ divided by $p|K|$. If $r<(p-1)|K|$, 
then $\omega(GP(q,d))=|K|$.    
\end{prop}

\begin{proof}
First, it is easy to verify that $K$ forms a clique in $GP(q,d)$ and thus $\omega(GP(q,d)) \geq |K|$; this was first observed in \cite{BDR88}. It suffices to show that $N=\omega(GP(q,d)) \leq |K|$. For the sake of contradiction, assume otherwise that $N \geq |K|+1$. Take $n=|K|+1$ so that $2 \leq n \leq N$. Since $r<(p-1)|K|$, there is no carrying between the addition of $r$ and $|K|$ in base-$p$. Thus, there is no carrying between the addition of $\frac{q-1}{d}$ and $|K|$ in base-$p$. It follows from Kummer's theorem that
$$
\binom{n-1+\frac{q-1}{d}}{\frac{q-1}{d}}=\binom{|K|+\frac{q-1}{d}}{\frac{q-1}{d}} \not \equiv 0 \pmod p.
$$
Thus, Theorem~\ref{thm:Nn} implies that
$$|K|(|K|+1) \leq (N-1)n \leq \frac{q-1}{d}<|K|(|K|+1),$$
a contradiction. This completes the proof that $\omega(GP(q,d))=|K|$.
\end{proof}

Theorem~\ref{thm:cliquenumber} is a special case of Proposition~\ref{prop:cliquenumber2}.

\begin{proof}[Proof of Theorem~\ref{thm:cliquenumber}]
Let $K=\F_p$. It suffices to check the assumptions in Proposition~\ref{prop:cliquenumber2}. Note that $d$ is a divisor of $\frac{q-1}{|K|-1}=p^2+p+1$ and $d|K|(|K|+1)>p^3=q$. Since $d>p$ and $d \mid (p^2+p+1)$, we have $d \geq p+2$ and thus $r=\frac{q-1}{d} \leq \frac{q-1}{p+2}<p^2-p$. 
\end{proof}

In the next examples, we construct several infinite families of generalized Paley graphs for which Proposition~\ref{prop:cliquenumber2} is applicable to determine their clique number. 
\begin{ex}
Let $q=p^{3m}$, $K=\F_{p^m}$, and $d=\frac{p^{2m}+p^m+1}{3}$, where $p$ is an odd prime with $p \equiv 1 \pmod 3$. Then $3d=\frac{q-1}{|K|-1}$ and thus $q<d|K|(|K|+1)$. Also note that $\frac{q-1}{d}=3(|K|-1)<(p-1)|K|$. Thus, Proposition~\ref{prop:cliquenumber2} implies that $\omega(GP(q,d))=|K|$.
\end{ex}

\begin{ex}
Let $s,t$ be coprime positive integers with $s>t$. Let $p$ be an odd prime, $q=p^{st}$, and $K=\F_{p^s}$. Since $s,t$ are coprime, we can take $d=\frac{(q-1)(p-1)}{(p^s-1)(p^t-1)} \in \Z$. Note that both $K$ and $\F_{p^t}$ are cliques in $GP(q,d)$. We deduce that $\omega(GP(q,d))=|K|=p^s$ from Proposition~\ref{prop:cliquenumber2}. Indeed, since $s>t$, we have $d|K|(|K|+1)>(q-1)(p-1)>q$. On the other hand, we also have
$$
\frac{q-1}{d}=\frac{(p^s-1)(p^t-1)}{p-1}=(p^s-1)(p^{t-1}+p^{t-2}+\cdots+p+1) \equiv p^s-(p^{t-1}+p^{t-2}+\cdots+p+1) \pmod {p^{s+1}}
$$
so that the reminder $r=p^s-(p^{t-1}+p^{t-2}+\cdots+p+1)<p^s=|K|$.
\end{ex}

\begin{ex}
Fix the degree of the field extension $\F_q$ over $K$. Given Proposition~\ref{prop:cliquenumber2}, it would be interesting to find a graph with edge density as large as possible (in other words, we want $d$ to be as small as possible) for which the assumptions in Proposition~\ref{prop:cliquenumber2} are satisfied, as Proposition~\ref{prop:cliquenumber2} is ``strongest" in that case. To achieve that, one needs to understand the distribution of divisors of $\frac{q-1}{|K|-1}$, which is, in general, difficult since $\frac{q-1}{|K|-1}$ is a polynomial in $|K|$, and $|K|$ is a prime power. 

We illustrate this goal in the special case $q=p^3$ and $K=\F_p$ in view of Theorem~\ref{thm:cliquenumber}. We need to find a divisor $d$ of $p^2+p+1$ such that $d>p$ and $d$ is as small as possible, that is, we want $d=p^{1+o(1)}$. Assume that $p$ is of the form $2x^2+x+1$; the infinitude of such primes would be guaranteed by Schinzel's hypothesis H \cite{SS58} or Bunyakovsky's conjecture. Observe that we have the identity
$$
(4x^2+3)(x^2+x+1)=(2x^2+x+1)^2+(2x^2+x+1)+1,
$$
so we can take $d=4x^2+3$ to be a divisor of $p^2+p+1$ such that $d \approx 2p$. 
\end{ex}

Next, we discuss the necessity of the two assumptions in the statement of Proposition~\ref{prop:cliquenumber2}.

\begin{ex}
Let $p$ be an odd prime, $q=p^4$, and $K=\F_p$. Proposition~\ref{prop:cliquenumber2} implies that $\omega(GP(q, 2(p^2+1)))=p$ and $K$ forms a maximum clique. However, note that $\omega(GP(q, p^2+1))=p^2$ since the subfield $\F_{p^2}$ forms a clique in $GP(q,p^2+1)$ and the trivial upper bound on the clique number is $p^2$. Note that $q<(p^2+1)p(p+1)$ while $\frac{q-1}{p^2+1}=p^2-1>(p-1)p$, which shows that the assumption ``$r<(p-1)|K|$" cannot be completely dropped in  Proposition~\ref{prop:cliquenumber2}. Interestingly, we have $\omega(GP(q, p^2+1))=p^2$ and $\omega(GP(q, 2(p^2+1)))=p$, while $GP(q, p^2+1)$ is simply the edge-disjoint union of two copies of $GP(q, 2(p^2+1))$.
\end{ex}

\begin{ex}
Let $q=5^6$, $d=3$, and $K=\F_{5^2}$. Note that $d \mid \frac{q-1}{|K|-1}=651$, so $K$ is a clique in $GP(q,3)$. Observe that $\frac{q-1}{d}=(1,3,1,3,1,3)$ in base-$5$, so $r<(p-1)|K|$. However, note that $d|K|(|K|+1)<q$. This shows that the assumption ``$q<d|K|(|K|+1)$" cannot be completely dropped in  Proposition~\ref{prop:cliquenumber2} since we in fact have $\omega(GP(q,3))=125$ since $\F_{125}$ forms a clique in $GP(q,3)$.
\end{ex}

\section*{Acknowledgment}
The author thanks Greg Martin, J\'ozsef Solymosi, Zixiang Xu, and Ethan White for helpful discussions. The author also thanks the anonymous referees for their valuable comments and suggestions.

\bibliographystyle{abbrv}
\bibliography{main}

\end{document}